\documentclass[11pt,final]{amsart}
\usepackage{amsmath,amssymb,amsthm,mathrsfs,showkeys,amsbsy, amsfonts,amscd}
\pdfoutput=1
\usepackage{graphicx}
\usepackage[usenames,dvipsnames]{pstricks} 
\usepackage{epsfig}
\usepackage[subnum]{cases}
\usepackage[colorlinks=true, linkcolor=magenta, citecolor=cyan, urlcolor=blue,hyperfootnotes=false]{hyperref}
\allowdisplaybreaks
\begin{document}
\newtheorem{lem}{Lemma}[section]
\newtheorem{prop}{Proposition}[section]
\newtheorem{cor}{Corollary}[section]
\numberwithin{equation}{section}
\newtheorem{thm}{Theorem}[section]
\theoremstyle{remark}
\newtheorem{example}{Example}[section]
\newtheorem*{ack}{Acknowledgment}
\theoremstyle{definition}
\newtheorem{definition}{Definition}
\theoremstyle{remark}
\newtheorem*{notation}{Notation}
\theoremstyle{remark}
\newtheorem{remark}{Remark}[section]
\newenvironment{Abstract} 
{\begin{center}\textbf{\footnotesize{Abstract}}%
\end{center} \begin{quote}\begin{footnotesize}}
{\end{footnotesize}\end{quote}\bigskip}
\newenvironment{nome}
{\begin{center}\textbf{{}}%
\end{center} \begin{quote}\end{quote}\bigskip}

\newcommand{\norm}[2]{{\left\| #1 \right\|}_{#2}}
\newcommand{\round}[1]{{\left ( #1 \right )}}
\newcommand{\abs}[1]{{\left | #1 \right |}}
\newcommand{\triple}[1]{{|\!|\!|#1|\!|\!|}}
\newcommand{\xx}{\langle x\rangle}
\newcommand{\ep}{\varepsilon}
\newcommand{\al}{\mu}
\newcommand{\be}{\beta}
\newcommand{\de}{\partial}
\newcommand{\la}{\lambda}
\newcommand{\La}{\Lambda}
\newcommand{\ga}{\gamma}
\newcommand{\del}{\delta}
\newcommand{\Del}{\Delta}
\newcommand{\sig}{\sigma}
\newcommand{\ome}{\omega}
\newcommand{\Ome}{\Omega}
\newcommand{\C}{{\mathbf C}}
\newcommand{\N}{{\mathbb N}}
\newcommand{\Z}{{\mathbf Z}}
\newcommand{\So}{{\mathbf S}}
\newcommand{\R}{{\mathbf R}}
\newcommand{\T}{{\mathbf T}}
\newcommand{\Rn}{{\mathbf R}^{n}}
\newcommand{\Rnu}{{\mathbb R}^{n+1}_{+}}
\newcommand{\Cn}{{\mathbb C}^{n}}
\newcommand{\spt}{\,\mathrm{supp}\,}
\newcommand{\M}{{\mathcal M}}
\newcommand{\Lin}{\mathcal{L}}
\newcommand{\SSS}{\mathcal{S}}
\newcommand{\F}{\mathcal{F}}
\newcommand{\xxi}{\langle\xi\rangle}
\newcommand{\eei}{\langle\eta\rangle}
\newcommand{\xei}{\langle\xi-\eta\rangle}
\newcommand{\yy}{\langle y\rangle}
\newcommand{\dint}{\int\!\!\int}
\newcommand{\hatp}{\widehat\psi}
\renewcommand{\Re}{\;\mathrm{Re}\;}
\renewcommand{\Im}{\;\mathrm{Im}\;}

\title[ NLS on $\R^n\times \M^k$]%
{{Well-posedness and scattering for the mass-energy NLS on $\R^n\times \M^k$}}
\author{}

\author{Mirko Tarulli}

\address{Mirko Tarulli\\
Dipartimento di Matematica, University of Pisa\\
56126 Pisa, Italy}

\email{tarulli@mail.dm.unipi.it}

\begin{abstract}
We study the nonlinear Schr\"odinger equation posed on product spaces $\R^n\times \M^k$, for $n\geq 1$ and  $k\geq1$, with $\M^k$ any $k$-dimensional compact Riemaniann manifold. 
The main results concern global well-posedness and scattering for small data solutions in non-isotropic Sobolev fractional spaces. In the particular case of $k=2$, $H^1$-scattering is also obtained. 
\end{abstract}

\maketitle

\date{}
\section{Introduction}
The main object of the paper is the study of nonlinear Schr\"odinger equation (NLS) on product spaces $\R^n\times \M^k$,
with $n\geq1$ and $\M^k$ is a compact Riemaniann $k$-manifold with  $k\geq1$. More precisely we consider the following family of Cauchy problems:
\begin{equation}\label{eq.lwanon}
\begin{cases}
i\partial_t u + \Delta_{x,y} u +\lambda| u|^\mu u=0, \ \ \ \ (t,x,y) \in \R\times\R^n\times \M^k\\
u(0,x,y)=f(x,y)\in \mathscr H^{0,\sigma}_{x,y},
\end{cases}
\end{equation}
for $\sigma>\frac k2$, where
$$
\Delta_{x,y}=\sum_{l=1}^n \partial_{x_l}^2+ \Del_{y}, 
$$ 
with $\Del_{y}$ the Laplace-Beltrami operator associated to the manifold $\M^k,$ defined in local coordinates by
$$
\frac 1{\sqrt{|g(y)|}}\partial_{y_i} \sqrt{|g(y)|}g^{hi}(y)\partial_{y_i},
$$
where $g_{hi}(y)$ is the metric tensor, $|g(y)|=\det (g_{hi}(y))$ and $g^{hi}=\left (g_{hi}(y) \right)^{-1}.$
Moreover $\lambda$ is any real number, the nonlinearity parameter $\mu$ satisfies the assumption
\begin{equation}\label{eq.al}
\frac 4n\leq\mu<\mu^*(n), \ \ \ \  \mu^*(n)=
\begin{cases}
\frac 4{n-1}  \ \ \ \  \  \  \text{if} \ \ \ n\geq2,\\
+\infty \ \ \ \ \ \,  \text{if} \ \ \ n=1,
\end{cases}
\end{equation}
and for any $s\in \R$, we denote
\begin{equation}\label{eq.AnSp}
\mathscr H^{s,\sigma}_{x,y} =
\round{1-\Delta_x}^{-\frac s2} \round{1-\Del_{y}}^{-\frac \sigma2}L^2_{x,y},
\end{equation}
where  $L_{x,y}^2=L^2(\Rn\times \M^k)$ (see \cite{Tri}).
Here we indicate $f\in L^2(\Rn \times \M^k)$ if the
\begin{align}\label{cons1}
\norm{f}{L^2(\Rn \times \M^k)}=\int_{\M^k}\int_{\Rn}|f|^2\, dxd\mathrm{v}_{g }<+\infty,
\end{align}
with $d \mathrm{v}_{ g}$ the volume element of $\M^k$ which reads in local coordinates as
$ \sqrt{|g(y)|}dy$. Furthermore the $h$-th component of the gradient operator $\nabla_y$ is given in local coordinates by $g^{hi}(y)\partial_{y_i}$.
Take into account that the power nonlinearity $\mu^*(n)$ corresponds to the  $H^1$-critical NLS in $\R^{n+1}$
as well as to the $H^{1/2}$-critical nonlinearity in $\R^n$.\\
Nonlinear equations of type \eqref{eq.lwanon} are deeply studied because their relevance from an applied science point of view. In fact, they describes the so called wave-guides important in the optics and communications theory (\cite{DanRa}, \cite{Sc} and \cite{SL}), as well as in quantum mechanics (see \cite{EK} and references therein), just to name a few. Consequently 
the analysis of the NLS on partially compact geometries has been the topic of many papers. In the particular case  of $\M^k=\T^k$, here $\T^k$ indicates the $k$-dimensional standard flat torus, we mention  \cite{TaTz} where it is studied the cubic NLS posed on $\R\times \T$ with initial data in $L^2$, \cite{IP} which handle with the energy-critical defocusing NLS on $\R\times \T^3$ and \cite{HTT} concerning global well posedness for solutions of \eqref{eq.lwanon} on $\R^n\times\T^{4-n}$ with cubic nonlinearity. We remand to \cite{HP}, \cite{TV2} in which the authors shed light also on the asymptotic behavior in the energy space for large data solutions of the defocusing NLS on $\R\times \T^2$ with quintic nonlinearity and on $\R^n\times \T^1, \, n\geq1$ with nonlinearity as in \eqref{eq.al} respectively. We mention also the remarkable paper \cite{HPTV} where modified scattering is obtained when cubic defocusing NLS is given on $\R\times \T^k, \, k=1,\dots,4$. Conversely for general product manifolds $\R^n\times \M^k,$ with $\M^k$ any compact manifold we quote \cite{TTV} where global well-posedness in $H^1_{x,y}(\R^n\times \M^1)$ is achieved for focusing NLS with nonlinearity such that $\mu<\frac4{n+1}$ and \cite{TV} where it is proved global well-posedness and scattering for the cubic NLS with data small with respect suitable non-isotropic Sobolev norms (actually larger than the classical energy $H^1$-norm when $k> 2$). 
In the above literature it arises, with some few exceptions, not only that to earn informations on well-posedness and scattering for the NLS given on product manifolds it is required an appropriate geometry for the compact manifolds, but also that  the asymptotic behavior of the solutions to \eqref{eq.lwanon} is poorly understood when the nonlinearity parameter $\mu$ such that $\frac 4n\leq \mu\leq \frac 4{n+k-2}$ is a fractional number.
\\
\\
Motivated by this, our main contribution in this paper is the well-posedness and scattering analysis of \eqref{eq.lwanon} with fractional pure power nonlinearities which satisfy \eqref{eq.al}, emphasizing that the only assumption we impose to the manifolds $\M^k$ is the compactness. We point out that the approach we used to study the local Cauchy problem associated with \eqref{eq.lwanon}  is divided in two different parts: we look first at $\mu=\frac 4n$ and then to $\mu<\frac 4{n-1}$. As far as concern the former case, that is
$L^2$-critical in $\R^n$, here we follows the spirit 
of the papers \cite{TV} (and also of \cite{TTV}).
Our key ingredient 
is a suitable version of Strichartz estimates of the type
$$\|e^{it\Delta_{x,y}} f\|_{L^q_t L^r_xL^2_y}\leq C \|f\|_{L^2_{x,y}},$$ where $(q, r)$ are such that
\begin{equation}\label{striKeelTao}
\frac 2q + \frac nr=\frac n2, \ \ q\geq 2, \ \ (n, q)\neq (2,2),
\end{equation}
(see \cite{KT} for more details), $L^2_y$ and $L^r_x$ are respectively the spaces $L^2(\M^k)$ and  $L^r(\R^n)$ and in 
general for any Banach space $X$ we define
$$\| f \|_{L^q_tX} =\left(
\int_{\R} \norm{f(x)}{X}^q dt\right)^{1/q},$$  (for its version local in time we adopt the symbol $L^q_{(t_1, t_2)} X$, with preassigned $t_1, t_2\in\R$).
Observe that the above estimates roughly speaking are a mixture of the (non dispersive) $L^2$-conservation w.r.t. the compact $y$ variable and
the classical Strichartz estimates w.r.t. the dispersive directions $\R^n$. Of course along with the above estimates one can also consider
similar ones  
for the Duhamel operator. Armed with those inequalities (and their version with derivatives) one can perform
a fixed point argument is the spaces $L^q_tL^r_xH^\sigma_y$,
where we denote, from now on, $H^{\sigma}_{y}=W^{\sigma, 2}_y$ with
\begin{equation}\label{sobalg}
L^r_xW^{\sigma, l}_{y}=\round{1-\Delta_y}^{-\frac \sigma2} L^r_x L^l_{y},
\end{equation}
 for $\sigma\in \R$ and where $L^l_y$ is the spaces $L^l(\M^k)$ for $l\geq1$.
The main advantage here is that one can consider the $\C$-valued solution $u(t,x,y)$ 
as functions dependent on the $(t, x)$ variables and valued in the algebra $H^\sigma_y$. Hence we have
in the product manifolds setting the same numerology involved in the study of NLS posed on $\R^n$
via admissible Strichartz norms $L^q_tL^r_x$, with $(q,r)$ as in \eqref{striKeelTao}, which enables also to transfer to $\R^n\times \M^k$ the scattering techniques available in $\R^n$.
On the other hand we notice that in the classical Euclidean theory the best nonlinearity that can be reached
with this technique is the $L^2$-critical in $\R^n$, i.e. $0<\mu\leq 4/n$. Namely the main difficulty in the transposition of the above analysis to nonlinearities which are $L^2$- supercritical and $H^{\frac 12}$-subcritical in $\R^n$, i.e. $4/n< \mu<\mu^*(n)$, is that in analogy with the analysis of $L^2$-supercritical NLS in $\R^n$ it seems to be necessary to work with Strichartz estimates involving derivatives
w.r.t. $x$ variable.\\
To overcome this obstacle we make an use, similarly to \cite{TV2}, of a class of inhomogeneous Strichartz estimates 
with respect the $x$ variable true in a range of Lebesgue exponents larger 
than the one given in the usual homogeneous estimates context (see \cite{Fo} and \cite{Vi}).
Observe also that following our discussion above we are interested to work with a space which is an algebra w.r.t. to $y$ variable, i.e.
$H^\sigma_y$ with $\sigma>k/2$, on the other hand
we need to consider at most $1/2$ derivatives w.r.t to $x$ variable since our nonlinearity
 is $H^{\frac 12}$ subcritical in $\R^n$, that is $\mu<\frac 4{n-1}$. Then the same argument used to deal with the $L^2$-critical case infers to well-posedness. In particular it also permits, by treating initial small data in $\mathscr H^{s,\sigma}_{x,y}$, to recover scattering in non-isotropic spaces $\mathscr H^{0,\sigma}_{x,y}$.
\\
\\
Next we enter in the heart of matter with the main contributions of this paper. The first results deals with the mass NLS.
\begin{thm}\label{halflwpWeak}
Assume $\mu=\frac 4n$. Then for every $n\geq 1$, $k\geq2$ and $\frac k2<\sigma<\frac 4n+1$ there exists a positive number $\varepsilon=\varepsilon(\sigma)$ such that the problem \eqref{eq.lwanon} enjoys a unique
global solution 
\begin{align}\label{lwp2a}
u(t,x,y)\in  L^\ell(\R; L^{\ell}_x(\Rn;  H_y^{\sigma}(\M^k)),
\end{align}
where $\ell=\frac{2n+4}{n}$, for initial data $f \in \mathscr H^{0,\sigma}_{x,y}$ such that $\|f\|_{\mathscr H^{0,\sigma}_{x,y}}<\varepsilon$. In addition $u(t,x,y)\in L^\infty( \R; \mathscr H^{0,\sigma}_{x,y})$ and there exist $\varphi_{\pm}  \in \mathscr H^{0,\sigma}_{x,y}$
so that
\begin{equation}\label{HSscatt}
\lim_{t\rightarrow \pm \infty} \|u(t,x,y) - e^{-it\Delta_{x,y}} \varphi_{\pm} \|_{\mathscr H^{0,\sigma}_{x,y}}=0.
\end{equation}
\end{thm}
\begin{remark}\label{H1}
In the above Theorem \ref{halflwpWeak} are presented peculiar properties of the solutions to \eqref{eq.lwanon}.
Specifically if we select in the Cauchy problem \eqref{eq.lwanon}, for  suitable $\epsilon>0$, initial data $f\in\mathscr H^{0,\frac k2+\epsilon}_{x,y}$ with size $\varepsilon$, one gets scattering 
in non-isotropic Sobolev spaces characterized as in \eqref{eq.AnSp} of which regularity is strictly connected with the spatial dimensions of the manifold $\M^k$ and the Euclidean part $\R^n$. Such an asymptotic behavior of the solutions $u(t,x,y)$ is completely independent from the geometry of $\M^k$,
which is required only to be a compact Riemannian manifold. Moreover, in the case $k=2$, we observe that if we take a sufficiently small $\epsilon>0$, then the space $\mathscr H^{0,1+\epsilon}_{x,y}$ is slightly stronger than $H^1_{x,y}$ only w.r.t. the $y$-variable. This fact is totally consistent with the results obtained in \cite{HP} and \cite{HTT}, for NLS with algebrical nonlinearities defined on  $\R\times\T^2$ and $\R^2\times\T^2$ respectively and where initial data in $H^1_{x,y}$ are considered. In addition we underline also that while scattering in energy space can not hold on $\R\times\T^3$ (see for \cite{IP} for example), we get on the same geometric setting the asymptotic completeness for $u(t,x,y)$ in non-isotropic spaces $\mathscr H^{0,\sigma}_{x,y}$.
\end{remark}


Then we can state the second result in the special case $k=2$, that is the mass-energy NLS. \begin{thm}\label{halflwpStrong}
Assume $\mu=\frac 4n$. Then for every $n\geq 1$ and $1<\sigma<1+\frac 4n$ there exists 
a positive number $\varepsilon=\varepsilon(\sigma)$ such that the problem \eqref{eq.lwanon} enjoys a unique
global solution 
\begin{align}\label{lwp2b}
u(t,x,y)\in L^\infty( \R;H^{1}_{x,y})\cap L^\ell( \R; L^{\ell}_x(\Rn;  H_y^{\sigma}(\M^2)),
\end{align}
where $\ell=\frac{2n+4}{n}$, in the following cases:
\begin{enumerate}
\item \label{defoc} if $\lambda<0$ (i.e. \eqref{eq.lwanon} is defocusing),  for initial data $f \in H_{x,y}^1\cap \mathscr H^{0,\sigma}_{x,y}$ such that $\|f\|_{\mathscr H^{0,\sigma}_{x,y}}<\varepsilon$;
\item \label{focus} if  $\lambda>0$, for initial data $f \in H_{x,y}^1\cap \mathscr H^{0,\sigma}_{x,y}$ such that $\|f\|_{H^1_{x,y}\cap\mathscr H^{0,\sigma}_{x,y}}<\varepsilon$ (i.e. 
\eqref{eq.lwanon} is focusing and the initial data are small). 
\end{enumerate}
In addition there exist $\varphi_{\pm}  \in H_{x,y}^1$
so that
\begin{equation}\label{H1scatt}
\lim_{t\rightarrow \pm \infty} \|u(t,x,y) - e^{-it\Delta_{x,y}} \varphi_{\pm} \|_{H^1_{x,y}}=0.
\end{equation}
\end{thm}

\begin{remark}
As we noticed, by pickung up $k=2$, the NLS in \eqref{eq.lwanon} becomes both mass and energy critical. In such a fashion the technicalities developed along the proof of the previous Theorem \ref{halflwpWeak} can be improved guaranteeing also well-posedness and scattering  in the energy space (other than $\mathscr H^{0,\sigma}_{x,y}$) if we consider additionally that initial data are $H^1_{x,y}$-bounded.
\end{remark}

At this point we give the following definition borrowed by \cite{Fo} (see also \cite{Vi})
\begin{definition}\label{accept}
We say that the pair $(q, r),$ is Schr\"odinger - acceptable  if
\begin{align}\label{StrFo}
1\leq q<\infty,\ \ 2\leq r\leq \infty, \ \  \frac 1q <n\left (\frac 12-\frac 1{r}\right ), \ \ \text{or} \ \ (q,r)=(\infty,2).
\end{align}
\end{definition}
Finally, for NLS having nonlinearity parameter in the remaining range arising from the condition \eqref{eq.al}, we can state the following:
\begin{thm}\label{halflwpWeakExt}
Assume $\frac 4n<\mu<\mu^*(n)$ and $s=\frac{\mu n -4}{2\mu}$. Then for any $n\geq 1$, $k\geq1$ and $\frac k2<\sigma<\min{\round{\frac {k+1}2-s, 1+\mu}}$,  (at least) a Schr\"odinger - acceptable couple $(q,r)$ and a positive number $\varepsilon=\varepsilon(\sigma)$ such that the problem \eqref{eq.lwanon} has a unique
global solution 
\begin{align}\label{lwp2c}
u(t,x,y)\in L^q( \R; L^{r}_x(\Rn; H_y^{\sigma}(\M^k)),
\end{align}
for initial data $f \in \mathscr H^{s,\sigma}_{x,y}$ and so that $\|f\|_{\mathscr H^{s,\sigma}_{x,y}}<\varepsilon$. In addition $u(t,x,y)\in L^\infty( \R; \mathscr H^{0,\sigma}_{x,y})$ and  there exist $\varphi_{\pm}  \in \mathscr H^{0,\sigma}_{x,y}$
such that
\begin{equation}\label{HSSscatt}
\lim_{t\rightarrow \pm \infty} \|u(t,x,y) - e^{-it\Delta_{x,y}} \varphi_{\pm} \|_{\mathscr H^{0,\sigma}_{x,y}}=0.
\end{equation}
\end{thm}
\begin{remark} 
Notice that if $k<1+2\mu$, then the above theorem remains valid also if one takes  $f\in \mathscr H^{s,\frac {k+1} 2}_{x,y}$. Furthermore the condition $\frac k2<\sigma< 1+\mu$ forces to restrict to specific integers $n,k\geq 1$ so that $k<\frac {2n+6}{n-1}$.
\end{remark}

We conclude the introduction by recalling other important known works regarding well-posedness about NLS posed on general manifolds $\M^k$.
For the energy-subcritical NLS settled on $\R^{n},$ we quote \cite{CW}, \cite{GV} 
and for the energy-critical NLS the fundamental papers \cite{bourgain} and \cite{Iteam}, which are just some of the examples among the contributions on the topic.
Considering NLS on compact Riemannian manifold, we quote  \cite{Bo} where it is treated the case of $\T^k$ and suitable nonlinearity, \cite{BGT1}
where it studied NLS on a general compact 3-manifold with $H^{\frac12}$-critical nonlinearity. In \cite{HTT1}
and \cite{IP1} it is analyzed the $H^1$-critical NLS on $\T^3$
and in \cite{PTW} the energy-critical NLS in $\So^3$.
We moreover quote \cite{B}, \cite{BCS} (here also scattering is obtained) and \cite{IPS} for NLS on the hyperbolic space.
\\
\\
{\bf Acknowledgments}:{\em The author is supported by the project FIRB 2012 Dispersive Dynamics: Fourier Analysis and Calculus of Variations.}

\section{Strichartz inequalities on $\Rn \times \M^k$}
This section is devoted to present the Strichartz inequalities for the free evolution operator
$e^{it\Del_{x,y}}$ associated to the free Schr\"odinger equation (that is, the equation in \eqref{eq.lwanon} with $\lambda=0$) and for its corresponding Duhamel operator. 
We indicated by $H_x^s$ the spaces $H^{s}(\Rn)$, moreover given any real number $1\leq p \leq \infty$ let be denoted by $1\leq p' \leq \infty$ its H\"older conjugate exponent.
Then one has the following:
\\
\begin{thm}\label{StriProd}
Let be $n\geq1,$ and $\sigma\in\R.$ Then the following homogeneous estimates hold
\begin{align}
\label{eq:200p}
 \|e^{ i t \Delta_{x,y}} f\|_{L^q_t L^r_xH_y^{\sigma}}\leq C\|f\|_{ \mathscr H^{s,\sigma}_{x,y}},
 \end{align}
when $(q,r)$ verify the condition
 \begin{align}
\label{eq:201p}  \frac{2}{q}+\frac{n}{r}= \frac{n}{2}-s, 
\end{align} 
with $s< \frac n2$ and satisfy the following
\begin{itemize}
\item  $4\leq q\leq\infty$ and $2\leq r\leq \infty,$ for $n=1;$
\item $2< q\leq\infty$ and $2\leq r\leq \infty,$ for $n=2;$
\item $2\leq q\leq\infty$ and $2\leq r\leq \infty,$ for $n\geq 3.$
\end{itemize}
Moreover the following extended inhomogeneous estimates hold
\begin{align}
 \label{eq:202p}
 \left \|\int_0^t e^{ i (t-\tau) \Delta_{x,y}} F(\tau,x, y)
d\tau\right\|_{L^q_t L^r_xH_y^{\sigma}}\leq C\norm{F}{L^{\tilde{q}'}_t
L^{\tilde{r}'}_xH_y^{\sigma}},
\end{align}
 when the Schr\"odinger-acceptable pairs $(q,r)$ and $(\tilde{q},\tilde{r})$ verify the condition
\begin{align}
\label{eq.203p}
 \frac{1}{q}+\frac{1}{\tilde{q}}=\frac{n}{2}\left (\frac{1}{\tilde{r}'}-\frac{1}{r}\right),
 \end{align}
  with $2\leq q, r\leq\infty$ and  $2\leq \tilde q, \tilde r\leq\infty$ and satisfy the following
\begin{itemize}
\item for $n=1,$ no additional conditions are needed;
\item  for $n=2,$ conditions $r<\infty$ and $\tilde{r}<\infty$ are required;
\item for $n\geq 3,$ the further conditions
\begin{align}
 \label{eq.204p}
 \frac{1}{q}<\frac{1}{\tilde{q}'}, \ \ \ \ \ \frac{n-2}{n}\leq \frac{r}{\tilde{r}}\leq\frac{n}{n-2},
\end{align} 
are needed.
\end{itemize}
\end{thm}

The proof of Theorem \ref{StriProd} follows by arguing as in the papers \cite{TV} and \cite{TV2}, we give it for the sake of completeness.
\begin{proof}
We introduce the operator

\begin{align*}
u(t,x,y)=e^{it\Delta_{x,y}}f(x,y)+  \int_{0}^{t} e^ {i
(t-\tau) \Delta_{x,y}} F(\tau,x,y) d\tau&,
\end{align*}
and observe that
\begin{align}\label{eq.cauchy}
i\partial_t u + \Delta_x u +\Delta_y u= F, \ \ \ \ (t,x,y)\in \R\times \R^n \times \mathcal M^k,
\end{align}
with $u(0,x,y)=f(x,y).$
We have the following decomposition with respect to the orthonormal basis $\{ \phi(y)\}$ of $L^2(\M^k)$ coming from the eigenfunctions of $-\Delta_y,$ that is $-\Delta_y \phi_k(y)=\nu_k \phi_k(y),$
\begin{align}
\label{dec1} u(t,x,y)=\sum_k u_k(t,x) \phi_k(y),&\\
\label{dec2} F(t,x,y)=\sum_k F_k(t,x) \phi_k(y),&\\
f(x,y)=\sum_k f_k(x) \phi_k(y).&\nonumber
\end{align}
Notice that the functions $u_k(t,x),$ $F_k(t,x)$ and $f_k(x)$ satisfy the following Cauchy problem
\begin{align}\label{eigen}
i\partial_t u_k + \Delta_x u_k  -\nu_k^2 u_k= F_k, \ \ \ \ (t,x)\in \R\times \R_x^n,
\end{align}
with $$u_k(0,x)=f_k(x).$$
Applying the classical homogeneous Strichartz estimates (see \cite{KT}) in the framework of \eqref{eigen} with $F_k(x,y)=0,$ we achieve
$$
\|e^{ i t (\Delta_{x}-\nu_k^2)}f_k\|_{L^q_t L^r_x}\leq C\|f_k\|_{H_x^s},
$$
with $(q, r)$ and $s$ as in \eqref{eq:200p}. We see that the presence of the unimodular factor $e^{-itk^2}$ does not affect the above estimates, moreover the constant $C>0$ does not depend on $k$.
Summing over $k$ the squares we arrive at
$$
\|e^{ i t (\Delta_{x}-\nu_k^2)}f_k\|_{\ell_k ^2 L^q_t L^r_x}\leq C\|f_k\|_{\ell_k ^2 H_x^s},
$$
furthermore, because of $q,r \geq 2,$ the Minkowski inequality yields
$$
\|e^{ i t (\Delta_{x}-\nu_k^2)}f_k\|_{ L^q_t L^r_x\ell_k ^2}\leq C\|f_k\|_{H_x^s\ell_k ^2}.
$$
By \eqref{dec1} and Plancherel identity we infer
$$
\|e^{ i t \Delta_{x,y}}f\|_{ L^q_t L^r_xL_y ^2}\leq C\|f\|_{H_x^sL_{y}^2},
$$
that is estimate \eqref{eq:200p} in the case $\sigma=0.$ Now it is sufficient to commute the equation 
\eqref{eq.cauchy} with the operator $(1-\Delta_y)^{\frac \sigma 2}$, then the estimate  \eqref{eq:200p} follows. 

If we pick up now $f_k(x,y)=0$ in the context of \eqref{eigen}, notice that we can proceed as above. In fact by an use of  the extended inhomogeneous Strichartz estimates (we remand again to \cite{Fo}) we get
\begin{align}
 \label{eq:202x} \left\|\int_0^t e^{ i (t-\tau) (\Delta_{x}-\nu_k^2)} F_k(s,\cdot)
d\tau\right\|_{L^q_t L^r_x}\leq C\norm{F_k}{L^{\tilde{q}'}_t
L^{\tilde{r}'}_x}.
\end{align}
with $(q, r)$ and $(\tilde{q},\tilde{ r})$ satisfying the relation \eqref{StrTV}.
Summing once again over $k$ the squares we have
\begin{align}
 \label{eq:202xy} \left\|\int_0^t e^{ i (t-\tau) (\Delta_{x}-\nu_k^2)} F_k(s,\cdot)
d\tau\right\|_{\ell_k ^2L^q_t L^r_x}\leq C\norm{F_k}{\ell_k ^2L^{\tilde{q}'}_t
L^{\tilde{r}'}_x},
\end{align}

Furthermore, because of $\max(\tilde{q}', \tilde{r}')\leq 2\leq \min(q,r),$ one can apply Minkowski inequality combined with Plancherel identity and then it is possible to write 
 \begin{align*}
\left\|\int_0^t e^{ i (t-\tau) \Delta_{x,y}} F_k(s,\cdot)
d\tau\right\|_{L^q_t L^r_x L_y^2}\leq C\norm{F_k}{L^{\tilde{q}'}_t
L^{\tilde{r}'}_xL^2_y}.
\end{align*} 
Finally, the classical argument of Fourier analysis theory  used in the proof of homogeneous estimates enhances that estimate \eqref{eq:202p} is proved. 
\end{proof}

We need also of an useful lemma deduced from \cite{TV} (see also \cite{TTV}). By this end we recall from
\cite{KT} the following:

\begin{definition}\label{Sadm}
An exponent pair $(\ell, p)$ is Schr\"odinger-admissible if $2\leq \ell,p\leq \infty,$  $(\ell, p, n ) \neq (2,\infty, 2),$ and
\begin{align}\label{StrTV}
\frac 2\ell +\frac n{p}=\frac n2.
\end{align}
\end{definition} 

Then we earn
\begin{lem}
Let be $n\geq1$ and indicate by $\mathcal D_i$ the operators $\mathcal D_1=\nabla_x$ and $\mathcal D_2=\nabla_y,$ 
then we have for $\gamma=0,1$ the following estimates
\begin{align}
\label{eq:200p1}
 \|\mathcal D^\gamma_i e^{ i t \Delta_{x,y}} f\|_{L^\ell_t L^p_xH_y^{\sigma}} + \left \|\mathcal D^\gamma_i \int_0^t e^{- i (t-\tau) \Delta} F(\tau,\cdot)
d\tau\right\|_{L^\ell_t L^p_x H_y^{\sigma}}&\nonumber\\
\leq C\|\mathcal D^\gamma_if\|_{ L^2_xH^\sigma_y}+\norm{\mathcal D^\gamma_iF}{L^{\tilde \ell'}_t
L^{\tilde p'}_xH_y^{\sigma}}&,
 \end{align}
 if for the pairs  $(\ell, p)$ and $(\tilde \ell, \tilde p)$ the condition \eqref{StrTV}
is satisfied for $\ell, \tilde \ell\geq2,$ if $n\geq3,$ $\ell, \tilde \ell>2$ if $n=2,$ and $\ell, \tilde \ell\geq4$ if $n=1$. 
Moreover we have also the estimates
\begin{align}\label{eq:200pb3}
\|\mathcal D^\gamma_ie^{ i t \Delta_{x,y}} f\|_{L^\infty_t L^2_xH_y^\sigma} +\left\|\mathcal D^\gamma_i\int_0^t  e^{- i (t-\tau) \Delta} F(s,\cdot)
d\tau\right\|_{L^\infty_t L^2_x H_y^{\sigma}}&\\
\leq C\|\mathcal D^\gamma_if\|_{ L^2_xH^\sigma_y}+ C\norm{\mathcal D^\gamma_iF}{L^{\tilde \ell'}_t
L^{\tilde p'}_xH_y^{\sigma}},&\nonumber
\end{align}
 with $(\ell,p)$ and $\gamma$ as above.
\end{lem}

\section{Proof of  Theorems  \ref{halflwpWeak} and \ref{halflwpStrong}}\label{pfmt}
This section, where it is assumed $\mu=\frac 4n$, has a double aim: firstly to show that to any choice of $n\geq1$, $k\geq2$ and small initial data in $\mathscr H^{0,\sigma}_{x,y}$ it is possible to prove that 
$u\in  L^{\ell}_{t}L^{ p}_{x}H^\sigma_y$ for any Schr\"odinger admissible pair $(\ell, p)$ by relying only on the structure of the Strichartz estimates. Secondly to display how to use the previous analysis to get  well-posedness in the spaces $L^\infty( \R; \mathscr H^{0,\sigma}_{x,y})$ (and $L^\infty( \R;H^{1}_{x,y})$ in the case $k=2)$ for the solutions to the Cauchy problem \eqref{eq.lwanon}, once one proves that the solutions lie in an auxiliary space needed to secure the closure of the fixed point argument. Finally we get scattering in $\mathscr H^{0,\sigma}_{x,y}$
(and in $H^{1}_{x,y}$ for $k=2$).

\begin{proof}[\bf{Proof of Theorem \ref{halflwpWeak}}]
Let be defined the integral operator associated to the Cauchy problem \eqref{eq.lwanon} with $\mu=\frac 4n$,
\begin{align}\label{eq.intop}
\mathcal{T}_fu=e^{it\Delta_{x,y}}f+  \lambda\int_{0}^{t} e^ {i
(t-\tau) \Delta_{x,y}} u(\tau)|u(\tau)|^\mu d\tau.
\end{align}
One needs to show that for $\overline \ell=\overline p=\mu+2$,
\begin{align*}
\forall f \in \mathscr H^{0,\sigma}_{x,y} \ \ \ \text{with} \ \ \|f\|_{\mathscr H^{0,\sigma}_{x,y}}<\varepsilon,\ \ \ \  \exists ! \   u(t,x,y)\in L^{\overline\ell}_{t}L^{\overline p}_{x}H^\sigma_y,
\end{align*}
satisfying the property
\begin{equation}\label{eq.fix}
\mathcal{T}_fu(t)=u(t).
\end{equation}
As well as we require 
$$
u(t,x,y)\in L^\infty(\R; \mathscr H^{0,\sigma}_{x,y})\cap L^{\overline\ell}_{t}L^{\overline p}_{x}H^\sigma_y.
$$
For simplicity, we split the proof in three further different steps.
\\
\\
\emph {Step One: for any $\sigma>\frac k2$, $\exists \, \varepsilon=\varepsilon(\sigma)>0$ and an $R=R(\sigma)>0$, such that  
\begin{equation}\label{embedd}
\mathcal{T}_f(B_{L^{\overline\ell}_{t}L^{\overline p}_{x}H^\sigma_y}(0,R))\subset B_{L^{\overline\ell}_{t}L^{\overline p}_{x}H^\sigma_y}(0,R),
\end{equation}
for any $f \in \mathscr H^{0,\sigma}_{x,y}$ so that $\|f\|_{\mathscr H^{0,\sigma}_{x,y}}<\varepsilon$}.
\\
\\
From now it is sufficient to deal directly with the case 
$(t_1,t_2)=\R$, specifying a different domain for the $t$-variable when it is required. We need to show \eqref{embedd}, to this end we start 
by applying the Sobolev fractional inequality \eqref{eq.LF} in combination with the embedding $H^\sigma_y\subset L^\infty_y$ which 
infers to the inequality
 \begin{align}\label{eq.1non0}
\|u|u|^\mu (t,x,\cdot)\|_{H_y^{\sigma}}  \leq
  \|u(t,x,\cdot)\|_{H_y^{\sigma}}^{\mu+1},
\end{align}
satisfied if $\frac k2<\sigma<1+\mu$ with $\mu> 0$.
Then, having in mind that $\mu=\frac 4n$, we can estimate the nonlinear term in the $L^{\overline\ell}_{t}L^{\overline p}_{x}H^\sigma_y$-norm by using the inhomogeneous Strichartz estimates in \eqref{eq:200p1} with $\gamma=0$, that is
 \begin{align}\label{eq.1nonl}
\|u|u|^\mu (t,x,\cdot)\|_{L^{\tilde{\ell}'}_t
L^{\tilde{p}'}_xH_y^{\sigma}}  \leq
  \| \|u(t,x,\cdot)\|_{H_y^{\sigma}}^{\mu+1}\|_{L^{\tilde{\ell}'}_tL^{\tilde{p}'}_x}.
\end{align}
Let us select 
\begin{align}\label{eq.condition1}
\frac{1}{\tilde{\ell}'}=\frac{\mu+1}{\overline\ell}=\frac{\mu+1}{\mu+2}, \ \ \  \frac{1}{\tilde{p}'}= \frac{(\mu+1)}{\overline p}=\frac{\mu+1}{\mu+2},
\end{align}
 then one observes that the r.h.s. of inequality \eqref{eq.1nonl}
 can be controlled by
\begin{align}\label{eq.2nonla}
  \| \|u(t,\cdot,\cdot)\|_{L^{\overline p}_x H_y^{\sigma}}^{\mu+1}\|_{L^{\tilde{\ell}'}_t}\leq\|u\|_{L^{(\mu+1)\tilde{\ell}'}_tL^{\overline p}_xH_y^{\sigma}}^{\mu+1}\leq\|u\|_{L^{\overline\ell}_tL^{\overline p}_x H_y^{\sigma}}^{\mu+1}.
  \end{align}
Thus we arrive at the following 
\begin{align}\label{eq.2nonl}
  \|\mathcal{T}_fu\|_{L^{\overline\ell}_tL^{\overline p}_x H_y^{\sigma}} \ \leq
    C \|f\|_{ \mathscr H^{0,\sigma}_{x,y}} + C \abs\lambda \|u\|_{L^{\overline\ell}_tL^{\overline p}_x H_y^{\sigma}}^{\mu+1}
\end{align}
and by a standard continuity argument (see for example Theorem 6.2.1 in \cite{Caz}) the previous estimate guarantees the existence of
an $\varepsilon>0$ and $R(\varepsilon)>0$ such that $\lim_{\varepsilon\rightarrow 0 }R(\varepsilon)=0$, provided that $\|f\|_{\mathscr H^{0,\sigma}_{x,y}}<\varepsilon$.
Additionally, according to the Remark \ref{StRem}, we get the following strong space-time bound
\begin{align}\label{eq.SpTim}
  \|u\|_{L^{\ell}_tL^{p}_x H_y^{\sigma}} \ \leq C\varepsilon
    <1,
\end{align}
with $C>0$, for the full set of Strichartz exponents $(\ell,p)$ given as in Definition \ref{Sadm}.
\\
\\
\emph {Step Two: $\mathcal{T}_f$ is a contraction on $B_{L^{\overline\ell}_{t}L^{\overline p}_{x}H^\sigma_y}(0,R),$ equipped
 with the norm $\norm{.}{L^{\overline\ell}_tL^{\overline p}_x L_y^{2}}.$}
\\
\\
 Given any $v_1,v_2\in B_{L^{\overline\ell}_{t}L^{\overline p}_{x}H^\sigma_y}$  we achieve, by a further use of the inhomogeneous estimate in \eqref{eq:200p1}, the chain of bounds
\begin{align}\label{eq.4nonl}
  \|\mathcal{T}_fv_1-\mathcal{T}_f v_2\|_{L^{\overline \ell}_t L^{\overline p}_xL_y^2} \leq  \abs\lambda \|v_1|v_1|^\mu-v_2|v_2|^\mu\|_{L^{\tilde{\ell}'}_t
L^{\tilde{p}'}_xL_y^2} \leq& \nonumber\\
\leq   \abs\lambda\norm{\|v_1-v_2\|_{L_y^2}(\norm{v_1}{L_y^\infty}^\mu+\norm{v_2}{L_y^\infty}^\mu)}{L^{\tilde{\ell}'}_t
L^{\tilde{p}'}_x}\leq&\nonumber\\
\leq  \abs\lambda\norm{\|v_1-v_2\|_{L^{\overline p}_xL_y^2}(\norm{v_1}{L^{\overline p}_x H^\sigma_y}^\mu+\norm{v_2}{L^{\overline p}_x H^\sigma_y}^\mu)}{L^{\tilde{\ell}'}_t}&,
\end{align}
where in the last inequality we used the second of the identities in \eqref{eq.condition1} and again the embedding $H^\sigma_y\subset L^\infty_y$. By Minkowski and H\"older inequalities the term in the third line of the previous \eqref{eq.4nonl}
can be bounded as follows

\begin{align}\label{eq.5nonl}
  \abs\lambda\norm{\|v_1-v_2\|_{L^{\overline p}_xL_y^2}(\norm{v_1}{L^{\overline p}_x  H_y^{\sigma}}^\mu+\norm{v_2}{L^{\overline p}_x  H_y^{\sigma}}^\mu)}{L^{\tilde{\ell}'}_t}\leq&\nonumber\\
 \leq C \abs\lambda \norm{\|v_1-v_2\|_{L^{\overline p}_xL_y^2}\norm{v_1}{L^{\overline p}_x  H_y^{\sigma}}^\mu}{L^{\tilde{\ell}'}_t}+\norm{\|v_1-v_2\|_{L^{\overline p}_xL_y^2}\norm{v_2}{L^{\overline p}_x  H_y^{\sigma}}^\mu}{L^{\tilde{\ell}'}_t}\leq&
 \nonumber\\
 \leq C \abs\lambda \|v_1-v_2\|_{L^{\overline \ell}_t L^{\overline p}_xL_y^2}\left (\norm{v_1}{L^{\overline \ell}_tL^{\overline p}_x  H_y^{\sigma}}^\mu+\norm{v_2}{L^{\overline \ell}_tL^{\overline p}_x  H_y^{\sigma}}^\mu\right),&
\end{align}
here the last inequality is a consequence now of the first of the identities in \eqref{eq.condition1}. Thus we arrive at

\begin{align}\label{eq.6nonl}
  \|\mathcal{T}_f(v_1-v_2)\|_{L^{\overline \ell}_t L^{\overline p}_xL_y^2}\\
\leq C \abs\lambda \|v_1-v_2\|_{L^{\overline \ell}_t L^{\overline p}_xL_y^2}\sup_{i=1,2}\left \{\norm{v_i}{L^{\overline \ell}_t L^{\overline p}_x H_y^{\sigma}} \right\}^\mu
\leq C \abs\lambda (R(\varepsilon))^\mu \|v_1-v_2\|_{L^{\overline \ell}_t L^{\overline p}_xL_y^2},
\nonumber
\end{align}
where in the last line of the chain of above inequalities we applied the bound \eqref{eq.SpTim}. Then $\mathcal{T}_f$ is a contraction provided that $\varepsilon>0$ is suitable small.
\\
\\
\emph {Step Three: the solution exists and it is unique in $L^{\overline\ell}_{t}L^{\overline p}_{x}H^\sigma_y$}. 
\\
\\
We are in position to show existence and uniqueness of the solution applying the contraction principle to the map $\mathcal T_f$ defined on the complete metric space $B_{L^{\overline\ell}_{t}L^{\overline p}_{x}H^\sigma_y}(0,R)$ and equipped with the topology induced by $\norm{.}{L^{\overline\ell}_{t}L^{\overline p}_xL_y^2}.$
\\
\\
\emph {Step Four: Regularity of the solution: proof of  $u(t,x,y)\in L^\infty(\R;\mathscr H^{0,\sigma}_{x,y}).$}
\\
\\
It is enough
to argue as in the previous steps just exploiting estimates \eqref{eq:200pb3} instead of  \eqref{eq:200p1} in the proof of \eqref{eq.2nonl}.  This observation, combined with estimate \eqref{eq.2nonl}, enhances to
\begin{align}\label{eq.2nonl3inf}
  \|\mathcal{T}_fu\|_{L_t^\infty \mathscr H^{0,\sigma}_{x,y}} \leq C \|f\|_{\mathscr H^{0,\sigma}_{x,y}} + C \abs\lambda \|u\|_{L^{\overline\ell}_tL^{\overline p}_x H_y^{ \sigma}}^{\mu}.
  \end{align}
   The above inequality with $\|f\|_{\mathscr H^{0,\sigma}_{x,y}}<\varepsilon$ and inequality \eqref{eq.SpTim} guarantee the fact that $u(t,x,y)\in L_t^\infty(\R;\mathscr H^{0,\sigma}_{x,y}).$
   \\
   \\
  The proof of the part of Theorem \ref{halflwpWeak} concerning the global well-posedness is accomplished. The remaining asymptotic completeness property \eqref{HSscatt} follows easily by standard arguments, we remand for instance to \cite{Caz} and \cite{Tao}.
   \end{proof}
   We are able now to give the following. 
  
  \begin{proof}[\bf{Proof of Theorem \ref{halflwpStrong}}]
From the proof of the previous Theorem \eqref{halflwpWeak} we already know that
there exists a unique solution to \eqref{eq.lwanon}, 
$u(t,x,y)\in L^{\overline\ell}_{t}L^{\overline p}_{x}H^\sigma_y$ once $k>\frac \sigma2$ and $\|f\|_{\mathscr H^{0,\sigma}_{x,y}}<\varepsilon$. 
Consider now the auxiliary norms
\begin{align}
\label{norm2}\| u\|_{\mathscr X_{t,x,y}^{(1)}(\ell,p)} = \sum_{k=0,1}\norm{\nabla^k_xu(t,x, y)}{{ L^\ell_tL^{p}_xL_y^2}},& \\ 
\label{norm3}\| u\|_{\mathscr X_{t,x,y}^{(2)}(\ell,p)} =  \sum_{k=0,1}\norm{\nabla^k_yu(t,x, y)}{{ L^\ell_tL^{p}_xL_y^2}},&
\end{align}
where $(\ell, p)$  are Schr\"odinger-admissible pairs. Then we can start by proving  the following.
\\
\\
\emph {Step One:  Let $u(t,x,y)$ be the unique solution to \eqref{eq.lwanon} with  
 initial data $f \in H_{x,y}^1\cap \mathscr H^{0,\sigma}_{x,y}$ such that $\|f\|_{\mathscr H^{0,\sigma}_{x,y}}<\varepsilon$. Then 
 \begin{equation}\label{sett}
 \| u(t,x,y)\|_{\mathscr X_{t,x,y}^{(1)}(\overline \ell,\overline p)}
 + \| u(t,x,y)\|_{\mathscr X_{t,x,y}^{(2)}(\overline \ell,\overline p)}
<\infty,
\end{equation}
with $\overline \ell=\overline p=\mu+2$.}
 \\
 \\
To this end we display that the classical Strichartz estimates \eqref{eq:200p1} in connection with the H\"older inequality yield for any $i=1,2$ and $(\tilde \ell', \tilde p')$ as in \eqref{eq.condition1},
\begin{align}
\label{eq.nonl3w}\|\mathcal{T}_fu\|_{\mathscr X_{t,x,y}^{(i)}(\ell,p)} \leq
\|\mathcal D^k_i f\|_{L^2_{x,y}} +  \abs\lambda\|\mathcal D^k_i (u|u|^\mu)\|_{L^{\tilde\ell'}_tL^{\tilde p'}_xL^2_y}&\\
\leq C \|\mathcal D^kf\|_{L_{x,y}^2} + C \abs\lambda\big \|\|\mathcal D^k_i u(t,x, y)\|_{L^2_y} \|u(t, x, y)\|_{L^\infty_y}^\mu \big \|_{L^{\tilde\ell'}_tL^{\tilde p'}_x}&\nonumber\\
\leq C \|D^kf\|_{L_{x,y}^2} +C \abs\lambda \big \|\|\mathcal D^k_iu(t,x, y)\|_{L^2_y} \|u(t, x, y)\|_{ H^\sigma_y}^\mu \big \|_{L^{\tilde\ell'}_tL^{\tilde p'}_x}&\nonumber\\
\leq C \|\mathcal D^kf\|_{L_{x,y}^2} + C \abs\lambda  \|\mathcal D^k_i u(t,x, y)\|_{L^{\overline\ell}_tL^{\overline p}_xL^2_y} \|u(t, x, y)\|_{L^{\overline\ell}_tL^{\overline p}_x H^\sigma_y}^\mu.&\nonumber
\end{align}

In that way we must have
\begin{align}\label{eq.2nonl2}
  \|\mathcal{T}_fu\|_{\mathscr X^{(i)}_{t,x,y}(\overline\ell,\overline p)}
   \leq C \|f\|_{H_{x,y}^1} + C \abs\lambda \|u\|_{\mathscr X^{(i)}_{t,x,y}(\overline\ell,\overline p)}\|u\|_{L^{\overline\ell}_tL^{\overline p}_x H^\sigma_y}^{\mu}\\\nonumber
   \leq C \|f\|_{H_{x,y}^1} + C (R(\varepsilon))^{\mu}\abs\lambda \|u\|_{\mathscr X^{(i)}_{t,x,y}(\overline\ell,\overline p)},
  \end{align}
  where in the second line we have used the bound \eqref{eq.SpTim}. The proof of \eqref{sett} is thus complete.
 \\
\\
\emph {Step Two: Regularity of the solution: proof of  $u(t,x,y)\in L^\infty( \R;H^{1}_{x,y}).$}
\\
\\
It is enough
to argue as in the previous step just using estimates \eqref{eq:200pb3} instead of  \eqref{eq:200p1} in the proof of \eqref{eq.nonl3w}. This fact gives
\begin{align}\label{eq.2nonl3}
  \|\mathcal{T}_fu\|_{L^\infty_{(t_1,t_2)}L^2_{x,y}} +\sum_{i=1,2}  \|\mathcal D_i\mathcal{T}_fu\|_{L^\infty_{(t_1,t_2)}L^2_{x,y}}\leq\nonumber\\
  C \|f\|_{H_{x,y}^1} + C \abs\lambda  \|u\|_{\mathscr X^{(i)}_{(t_1,t_2),x,y}(\overline\ell,\overline p)}\|u\|_{L^{\overline\ell}_{(t_1,t_2)}L^{\overline p}_x H^\sigma_y}^{\mu},
  \end{align}
  An use of \eqref{eq.SpTim} and \eqref{sett} provided that $\|f\|_{\mathscr H^{0,\sigma}_{x,y}}<\varepsilon$ allows to take $(t_1,t_2)=\R$ and enhances to $u(t,x,y)\in L^\infty_tH^{1}_{x,y}.$
\\
\\
As a straightforward consequence we have \eqref{lwp2b} in the case of \eqref{defoc}. In the focusing case we are forced to necessitate in \eqref{focus} also that  $\|f\|_{H^1_{x,y}}<\varepsilon$, in order to avoid some blow-up phenomena, as noticed in the paper \cite{KM}.
 \\
  \\
 It remains to establish the asymptotic completeness property \eqref{H1scatt}. It follows by applying a standard argument (see \cite{Caz}).
In fact by using the integral equation associated with \eqref{eq.lwanon} it is sufficient to prove that
\begin{equation}\label{kdv}\lim_{t_1, t_2\rightarrow \infty}
\|\int_{t_1}^{t_2} e^{-i s\Delta_{x, y}} (u|u|^\mu) ds\|_{H^1_{x,y}}=0.
\end{equation}
From the dual estimate to the homogeneous inequality in \eqref{eq:200p1}
we get:
\begin{equation}\label{dual}
\|\int_{t_1}^{t_2} e^{-i s\Delta_{x, y}} F(s) ds\|_{L^2_{x,y}}
\leq C \|F\|_{L^{\tilde \ell'}_{(t_1, t_2)} L^{\tilde p'}_x L^2_y},
\end{equation}
where $(\tilde \ell', \tilde p')$ are as in \eqref{eq.condition1}. Hence
\eqref{kdv} follows if one earns 
\begin{align}
\lim_{t_1, t_2\rightarrow \infty} \|u|u|^\mu\|_{L^{\tilde\ell'}_{(t_1, t_2)} L^{\tilde p'}_x L^2_y}\\
+\lim_{t_1, t_2\rightarrow \infty}\big(
\|\nabla_y (u|u|^\mu)\|_{L^{\tilde\ell'}_{(t_1, t_2)} L^{\tilde p'}_x L^2_y}
+ \|\nabla_x (u|u|^\mu)\|_{L^{\tilde\ell'}_{(t_1, t_2)} L^{\tilde p'}_x L^2_y}\big)=0.
\nonumber
\end{align}
The above limit can be proved following the same argument used along the proof
of the previous steps, in conjunction with the fact that $u(t,x,y)\in L^{\overline\ell}_{t}L^{\overline p}_{x}H^\sigma_y$ and \eqref{sett}.

\end{proof}
\begin{remark}\label{StRem}
We underline here that the unique global solution $u(t,x,y)\in  L^{\mu+2}_t(\R; L^{\mu+2}_x(\Rn\times  H_y^{\sigma}(\M^k))$ earned in Theorems \ref{halflwpWeak} (and consequently in Theorem
\ref{halflwpStrong})
fulfills also
$$u(t,x,y)\in  L^{\ell}_t(\R; L^{p}_x(\Rn\times  H_y^{\sigma}(\M^k))$$
for the full set of Strichartz exponents $(\ell,p)$ as in Definition \ref{Sadm} because of \eqref{eq.SpTim}.
\end{remark}

\section{Proof of Theorem \ref{halflwpWeakExt}}

In this section we present the global well-posedness and scattering results when 
the range of the nonlinear power is slightly enlarged to $\frac 4n<\mu<\mu^*(n)$.
In this regime, we display the use of the extended inhomogeneous Strichartz estimates
\eqref{eq:202p}.
We need some preliminaries before to give the proof of the theorem. For this purpose we 
introduce the following:

\begin{prop}\label{AlgebricCauchy}
Let $n\geq 3$, $4/n\leq \mu<4/(n-1)$ be fixed and
$s=\frac{\mu n -4}{2\mu}$.
Then there exist
$(q, r, \tilde{q}, \tilde{r}) $
such that:
\begin{equation}\label{easyconitionv3}0<\frac 1q,\frac 1r, \frac 1{\tilde q}, \frac 1{\tilde r}<\frac 12
\end{equation} and
\begin{align}
\label{eq.ass3v3}\frac 1 q+ \frac 1{\tilde q}<1, \quad \quad 
& \frac{n-2}{n}<\frac{r}{\tilde{r}}< \frac{n}{n-2}\\
\label{dropd12}
\frac{1}{q}+\frac{n}{r}<\frac{n}{2}, \quad \quad & 
\frac{1}{\tilde{q}}+\frac{d}{\tilde{r}}<\frac{n}{2}\\
\label{eq.ass102v3}  \frac{2}{q} +\frac{n}{r}=\frac n2- s,\quad \quad &
\frac{2}{q} + \frac nr + \frac{2}{\tilde{q}} +\frac n{\tilde r}=n,\\
\label{eq.ass9bisv} \frac{1}{\tilde{q}'}=\frac{\mu+1}{q}, \quad \quad& \frac{1}{\tilde{r}'}=\frac{\mu+1}{r}.&
\end{align}
For $n=1,2$ we get the same conclusion, provided that we drop 
conditions \eqref{eq.ass3v3}. We can also assume that
\begin{equation}\label{impBosUl3_gif}
\frac{\mu}q+\frac{\mu n}{2r}=1,\quad \quad \ \frac \mu r<1.
\end{equation}
\end{prop} 
\begin{proof}
This Lemma is nothing else that Lemma 8.2 in \cite{TV2}, so we skip.
\end{proof}

Notice that the previous proposition allows a combined use of \eqref{eq:200p} and \eqref{eq:202p}, this fact means that we have the estimate
\begin{align}\label{eq:202p90}
 &\|e^{-it\Delta_{x,y}} f\|_{L^{q}_t L^{r}_xH_y^{\sigma}} +\left \|\int_0^t e^{ -i (t-\tau) \Delta_{x,y}} F(\tau)
d\tau\right\|_{L^{q}_t L^{r}_xH_y^{\sigma}}\\
\nonumber &\leq C(\|f\|_{ H_x^{s}H_y^{\sigma}}+\|F\|_{L^{\tilde{q}'}_t
L^{\tilde{r}'}_xH_y^{\sigma}})
\end{align}
for every $\sigma\in \R$.
We will also need of the next result.
\begin{lem}\label{sTrbis}
Let $n\geq 1$ and $4/n\leq \mu<4/(n-1)$ be fixed.
Then there exist a Schr\"odinger-admissible pair $(\ell,p)$
such that:
\begin{align}
\label{st2bis}&\frac 1{p'}=\frac 1p+\frac \mu r,\\
\label{st3bis}&\frac 1{\ell'}=\frac 1\ell+ \frac \mu q,
\end{align}
where $(q,r)$ is any couple given by Proposition
\ref{AlgebricCauchy}.
\end{lem}
\begin{proof}
Also here we have that above lemma is Proposition 3.3 in \cite{TV2}, with the following minor modification: the identities \eqref{StrTV} and \eqref{st2bis} imply
$$\frac 1\ell= \frac {\mu n}{4r},\quad \quad \frac 1p= \frac 12 - \frac \mu{2r}$$
and thus the condition 
\eqref{st3bis} becomes
\begin{equation}\label{usef}\frac{\mu n}{2r} +\frac{\mu}{q}=1,\
\end{equation}
which is verified by \eqref{impBosUl3_gif}.
\end{proof}

\begin{proof}[\bf{Proof of Theorem \ref{halflwpWeakExt}}]
Let us look first for the global well-posedness. We want to show that, for any $(q,r)$ as in Proposition \ref{AlgebricCauchy},
\begin{align*}
\forall f \in \mathscr H^{s,\sigma}_{x,y} \  \ \text{with}\ \ 0<s<\frac12 \ \ \text{and} \ \|f\|_{\mathscr H^{s,\sigma}_{x,y}}<\varepsilon, \ \ \  \exists ! \   u(t,x,y)\in L^{q}_t L^{r}_xH_y^{\sigma}
\end{align*}
and that the operator $\mathcal T_f$ defined in 
 \eqref{eq.intop} satisfies \eqref{eq.fix}. We require also for 
$
u(t,x,y)\in L_t^\infty( \R; \mathscr H^{0,\sigma}_{x,y}).
$

\vspace{0.2cm}
We proceed as usual by splitting the proof in further different steps. 
  Consider now the further resolution norm
 \begin{align}\label{eq.1st2}
\norm{w}{\mathcal Z^\sigma(q,r)}=\norm{w}{L^{q}_t L^{r}_xH_y^{\sigma}}+\norm{w}{L^{\ell}_t L^{p}_xH_y^{\sigma}},
\end{align}
where $(\ell, p)=(\ell(q,r),p(q,r))$ is the pair uniquely determined by the conditions \eqref{st2bis} and \eqref{st3bis} of Lemma \ref{sTrbis}. Then, similarly to the proof of Theorem \ref{halflwpWeak}, the first task consists in proving what follows.
\\
\\
\emph {Step One: for any $\sigma>\frac k2$, $\exists \, \varepsilon=\varepsilon(\sigma)>0$ and an $R=R(\sigma)>0$, such that  
\begin{equation*}
\mathcal{T}_f(B_{\mathcal Z^\sigma(q,r)}(0,R))\subset B_{\mathcal Z^\sigma(q,r)}(0,R),
\end{equation*}
for any $f \in \mathscr H^{s,\sigma}_{x,y}$ with $0<s<\frac12$ as in Proposition \ref{AlgebricCauchy} and so that $\|f\|_{\mathscr H^{s,\sigma}_{x,y}}<\varepsilon$.}
\\
\\
We need to deal first with the control of the nonlinear term in the $L^{q}_t L^{r}_xH_y^{\sigma}$-norm. The inhomogeneous Strichartz estimates contained in  \eqref{eq:202p90} in connection again with the fractional inequality \eqref{eq.1non0}, now with $\mu<\mu^*(n)$,  bring to
 \begin{align}\label{eq.1nonlEx}
\|u|u|^\mu (t,x,\cdot)\|_{L^{\tilde{q}'}_t
L^{\tilde{r}'}_xH_y^{\sigma}}  \leq
  \| \|u(t,x,\cdot)\|_{H_y^{\sigma}}^{\mu+1}\|_{L^{\tilde{q}'}_tL^{\tilde{r}'}_x}.
\end{align}
We select 
\begin{align*}
\frac{1}{\tilde{r}'}=\frac{\mu+1}{r}, \ \ \  \frac{1}{\tilde{q}'}=\frac{(\mu+1)}{q},
\end{align*}
as in \eqref{eq.ass9bisv} and with $(q,r, \tilde{q}, \tilde{r})$ a point  allowed  by Proposition \ref{AlgebricCauchy}. An application of H\"older inequality gives that the l.h.s of \eqref{eq.1nonlEx}
 can be bounded by
\begin{align}\label{eq.2nonlaEx}
  \| \|u(t,\cdot,\cdot)\|_{L^{r}_x H_y^{\sigma}}^{\mu+1}\|_{L^{\tilde{q}'}_t}\leq\|u\|_{L^{(\mu+1)\tilde{q}'}_tL^{r}_x H_y^{\sigma}}^{\mu+1}\leq\nonumber\\
  \leq C\|u\|_{L^{q}_tL^{r}_x H_y^{\sigma}}^{\mu+1},
\end{align}
for some constant $C>0$. Finally we can write the following
\begin{align}\label{eq.2nonlEx}
  \|\mathcal{T}_fu\|_{L^{q}_t L^{r}_xH_y^{\sigma}} \ \leq
    C \|f\|_{H_x^{s} H_y^{\sigma}} + C  \abs\lambda \|u\|_{L^{q}_t L^{r}_xH_y^{\sigma}}^{\mu+1},
\end{align}
 with $0<s<\frac 12$ as in Proposition \ref{AlgebricCauchy}.
Furthermore the classical Strichartz estimate \eqref{eq:200p1} in connection with Lemma \ref{sTrbis} and an application of H\"older inequality
yield the following
\begin{align}
\label{eq.nonl3Ex}\|\mathcal{T}_fu\|_{L^{\ell}_t L^{p}_xH_y^{\sigma}} \leq
\| f\|_{\mathscr H^{0,\sigma}_{x,y}} +  \abs\lambda\|u|u|^\mu\|_{L^{\ell'}_tL^{p'}_xH_y^{\sigma}}&\\
\leq C \|f\|_{\mathscr H^{0,\sigma}_{x,y}} + C \abs\lambda\big \| \|u(t, x, y)\|_{H_y^{\sigma}}^{\mu+1} \big \|_{L^{\ell'}_TL^{p'}_x}&\nonumber\\
\leq C \|f\|_{\mathscr H^{0,\sigma}_{x,y}} +C \abs\lambda \big \|\|u(t,x, y)\|_{H_y^{\sigma}} \|u(t, x, y)\|_{ H^\sigma_y}^\mu \big \|_{L^{\ell'}_tL^{p'}_x}&\nonumber\\
\leq C \|f\|_{\mathscr H^{s,\sigma}_{x,y}} + C  \abs\lambda \| u(t,x, y)\|_{L^{\ell}_tL^{p}_xH_y^{\sigma}} \|u(t, x, y)\|_{L^{q}_tL^{r}_x H^\sigma_y}^\mu,&\nonumber
\end{align}
where in the last inequality we used, for any $s\geq0$, the embedding $\mathscr H^{s,\sigma}_{x,y}\subset \mathscr H^{0,\sigma}_{x,y}$, which is a direct consequence of Fubini's Theorem, Plancherel's identity w.r.t. the $x$-variable and the fact that the operators $\round{1-\Delta_x}^{-\frac s2} $ and $\round{1-\Del_{y}}^{-\frac \sigma2}$ commute. In that way we must have
\begin{align}\label{eq.2nonl2Ex}
  \|\mathcal{T}_fu\|_{L^{\ell}_t L^{p}_xH_y^{\sigma}} \leq
  C \|f\|_{\mathscr H^{s,\sigma}_{x,y}} + C \abs\lambda \|u\|_{L^{\ell}_t L^{p}_xH_y^{\sigma}}\|u\|_{L^{q}_tL^{r}_x H_y^{\sigma}}^{\mu}.
  \end{align}
A combination of \eqref{eq.2nonlEx} and  \eqref{eq.2nonl2Ex} brings to the inequality
  \begin{align}\label{eq.3nonl2Ex}
  \|\mathcal{T}_fu\|_{\mathcal Z^\sigma(q,r)} \leq
  C \|f\|_{\mathscr H^{s,\sigma}_{x,y}} + C \abs\lambda \|u\|_{\mathcal Z^\sigma(q,r)}\|u\|_{L^{q}_tL^{r}_x H_y^{\sigma}}^{\mu},
  \end{align}
  then it is enough to proceed as in the proof of Theorem \eqref{halflwpWeak}.
\\
\\
We have to approach now the further step.
\\
\\
\emph {Step Two: $\mathcal{T}_f$ is a contraction on $B_{\mathcal Z^\sigma(q,r)}(0,R),$ equipped
 with the norm $\norm{.}{L^{q}_tL^{r}_x L_y^{2}}.$}
 \\
 \\
 Given any $v_1,v_2\in B_{X^{\sigma}_T(q,r)}(0,R)$  we achieve, again by an use of estimates \eqref{eq:202p}, the chain of bounds
\begin{align}\label{eq.4nonlEx}
  \|\mathcal{T}_fv_1-\mathcal{T}_f v_2\|_{L^q_t L^r_x L_y^2} \leq \abs\lambda \|v_1|v_1|^\mu-v_2|v_2|^\mu\|_{L^{\tilde{q}'}_t
L^{\tilde{r}'}_x L_y^2} \leq& \nonumber\\
\leq \abs\lambda\norm{\|v_1-v_2\|_{L_y^{2}}(\norm{v_1}{ H_y^{\sigma}}^\mu+\norm{v_2}{ H_y^{\sigma}}^\mu)}{L^{\tilde{q}'}_t
L^{\tilde{r}'}_x}\leq&\nonumber\\
\leq  \abs\lambda\norm{\|v_1-v_2\|_{L^{r}_x L_y^2}(\norm{v_1}{L^{r}_x  H_y^{\sigma}}^\mu+\norm{v_2}{L^{r}_x  H_y^{\sigma}}^\mu)}{L^{\tilde{q}'}_t}&,
\end{align}
by Minkowski and H\"older inequalities we see that the last term in previous estimate \eqref{eq.4nonlEx}
can be controlled as follows

\begin{align}\label{eq.5nonlEx}
 \abs\lambda \norm{\|v_1-v_2\|_{L^{r}_x L_y^2}(\norm{v_1}{L^{r}_x  H_y^{\sigma}}^\mu+\norm{v_2}{L^{r}_x  H_y^{\sigma}}^\mu)}{L^{\tilde{q}'}_t}\leq&\nonumber\\
 \leq C \abs\lambda \norm{\|v_1-v_2\|_{L^{r}_x L_y^2}\norm{v_1}{L^{r}_x  H_y^{\sigma}}^\mu}{L^{\tilde{q}'}_t}+\norm{\|v_1-v_2\|_{L^{r}_x L_y^2}\norm{v_2}{L^{r}_x  H_y^{\sigma}}^\mu}{L^{\tilde{q}'}_t}\leq&
 \nonumber\\
 \leq C \abs\lambda \left (\norm{v_1}{L^{q}_tL^{r}_x  H_y^{\sigma}}^\mu+\norm{v_2}{L^{q}_tL^{r}_x  H_y^{\sigma}}^\mu\right)\|v_1-v_2\|_{L^{q}_t L^{r}_x L_y^2}.&
\end{align}
Thus we arrive at

\begin{align}\label{eq.6nonlEx}
  \|\mathcal{T}_f(v_1-v_2)\|_{L^q_t L^r_x L_y^2}\\
\leq C \abs\lambda \|v_1-v_2\|_{L^q_t L^r_x L_y^2}\sup_{i=1,2}\left \{\norm{v_i}{L^{q}_tL^{r}_x  H_y^{\sigma}} \right\}^\mu\leq C \abs\lambda \|v_1-v_2\|_{L^q_t L^r_x L_y^2}(R(\varepsilon))^\mu\nonumber
\end{align}
and again $\lim_{\varepsilon\rightarrow 0 }R(\varepsilon)=0$ in the above estimate allows to conclude.
\\
\\
\emph {Step Three: the solution exists and it is unique in $\mathcal Z^\sigma(q,r).$ Furthermore we
have $u(t,x,y)\in L_t^\infty( \R; \mathscr H^{0,\sigma}_{x,y})$}. 
\\
\\
We apply the contraction principle to the map $\mathcal T_f$ defined on the complete metric space $B_{\mathcal Z^\sigma(q,r)}(0,R)$ and equipped with the topology induced by $\norm{.}{L^q_t L^r_x L_y^2}.$ Moreover to guarantees $u(t,x,y)\in L_t^\infty( \R; \mathscr H^{0,\sigma}_{x,y})$
it suffices, by arguing as in the previous steps just using estimates \eqref{eq:200pb3} instead of  \eqref{eq:200p1} in the proof of \eqref{eq.nonl3Ex}, the inequality

\begin{align}\label{eq.2nonl3Ex}
  \|\mathcal{T}_fu\|_{L_t^\infty \mathscr H^{0,\sigma}_{x,y}} \leq
  C \|f\|_{\mathscr H^{s,\sigma}_{x,y}} + C \abs\lambda  \|u\|_{\mathcal Z(q,r)}\|u\|_{L^{q}_tL^{r}_x  H_y^{\sigma}}^{\mu}.
  \end{align}
\\
\\
To complete the proof of Theorem \ref{halflwpWeakExt}, it remains to see for the asymptotic completeness property \eqref{HSSscatt}. It follows, in a standard fashion, by proving that
\begin{equation}\label{kdvEx}\lim_{t_1, t_2\rightarrow \infty}
\|\int_{t_1}^{t_2} e^{-i s\Delta_{x, y}} (u|u|^\mu) ds\|_{\mathscr H^{0,\sigma}_{x,y}}=0
\end{equation}
From the dual estimate \eqref{dual}
with $( \ell', p')$ chosen as in Lemma \ref{sTrbis}, we have that 
\eqref{kdvEx} is a consequence of  
\begin{align}
\lim_{t_1, t_2\rightarrow \infty} \|u|u|^\mu\|_{L^{\ell'}_{(t_1, t_2)} L^{p'}_x H_y^\sigma}=0.
\nonumber
\end{align}
Such a limit can be earned following the same scheme of the proof
of \eqref{eq.nonl3Ex}, in conjunction with the fact that $u(t,x,y)\in L^{q}_tL^{r}_x H_y^{\sigma}\cap L^{\ell}_{t}L^{p}_{x}H^\sigma_y$.

\end{proof}
\appendix

\section{A fractional inequality on compact manifolds} 
The target of this Appendix, having its own interest, is to present fundamental tools giving the way, in the end, to the proof of inequality \eqref{eq.1non0}.
Given any compact manifold $\M^k$, we invoke the following basics:
\begin{itemize}
\item the curvature tensor (with its derivatives) is bounded;
\item the Ricci curvature tensor is bounded from below;
\item the injectivity radius is positive. 
\end{itemize}

These facts enable to represent the fractional derivative 
$|\nabla_y|^\sigma=\round{-\Del_{y}}^{\frac \sigma2}$ when $0<\sigma<1$ as 
\begin{align}\label{eq.fract}
|\nabla_y|^\sigma f(x)=\round{\int_0^\infty\round{\frac 1{t^\sigma \mathrm{v}_{g}(B(x,t))}\int_{B(x,t)}\left | f(x)-f(y) \right |d\mathrm{v}_g(y)}^2\, \frac{dt}t}^{\frac 12},
\end{align}
where with $B(x,t)$ we indicate the open ball of center $x\in \M^k$ and radius $t>0$ (for additional details we remand to \cite{BBR} or \cite{CRT}).
Therefore we can recall the next
\begin{lem}\label{HC}
Assume $\M^k$ a compact manifold with dimension $k\geq1$ and let $\phi$ be an H\"older continuous function of order $0<\mu<1$. Thus, for any $0<s<\mu$, $1<q<\infty$ and $\frac s\mu<\sigma<1$ we get
\begin{equation}
\norm{|\nabla_y|^s \phi(f)}{L^q}\leq C \norm{|f|^{\mu-\frac s\sigma}}{L^{q_1}}\norm{|\nabla_y|^\sigma f}{L^{{q_2}\frac s\sigma}}^{\frac s\sigma},
\end{equation}\label{eq.leib}
with $C>0$, $\frac 1q=\frac 1{q_1}+\frac 1{q_2}$ and $\round{1-\frac s{\mu\sigma}}q_1>1$.
\end{lem}
\begin{proof}
The proof is the same as in Proposition A.$1$ in \cite{Visa} and works in our framework without any changes. It comes out from the pointwise
inequality 
\begin{equation}\label{eq.max}
|\nabla_y|^s \phi(f)(x)\leq C\round{M(|f|^\mu)(x)}^{1-\frac s{\mu\sigma}}\round{|\nabla_y|^sf(x)}^{\frac s\sigma},
\end{equation}
where 
$$
M(f)(x)=\sup_{t>0}\frac 1{\mathrm{v}_g(B(x,t))}\int_{B(x,t)}\left | f(y) \right |d\mathrm{v}_g(y),
$$
is the Hardy-Littlewood maximal operator defined on $\M^k$ (for additional details we refer \cite{BS} and
\cite{CoWe}).
\end{proof}

At this point we can shape the main result of this section (we refer to \cite{CRT}, see also \cite{GK} and \cite{RS}
for an analogue property on the flat manifold $\R^n$). 

\begin{prop}\label{LF}
For any $f\in H_y^{\sigma}\cap L_y^\infty$ let $G(f)=f|f|^\mu$ be a real function with $\mu>0$. Then 
one has
\begin{equation}\label{eq.LF}
\norm{f|f|^\mu}{H_y^{\sigma}}\leq C \norm{f}{H_y^{\sigma}}\norm{f}{L_y^{\infty}}^{\mu},
\end{equation}
with $C>0$, provided that $0<\sigma<1+\mu$.
\end{prop}
\begin{proof}
\emph{Case $0<\sigma<1$.} Because in this regime one has 
\begin{equation}\label{def.fracSob}
 \norm{f}{H_y^{\sigma}}\sim \norm{f}{L_y^{2}} +\norm{|\nabla_y|^\sigma f}{L_y^{2}},
\end{equation}
with $|\nabla_y|^\sigma $ as in \eqref{eq.fract}, the result is given 
by an application of the elementary inequality
\begin{align*}
\left | f(x)|f(x)|^\mu-f(y)|f(y)|^\mu \right |
\\ \leq C\norm{|f|^\mu}{L_y^{\infty}}\left | f(x)-f(y) \right |
\leq C\norm{f}{L_y^{\infty}}^{\mu}\left | f(x)-f(y) \right |,
\end{align*}
(for the the proof of \eqref{eq.LF} in the specific case of $\M^k=\T^1$ we remand to Lemma $4.1$ in \cite{TV2}).
\\
\emph{Case $\sigma=1$.} This is given by the fact that the $L^\infty\cap H_y^{\sigma}$ is an algebra.
\\
\emph{Case $\sigma>1$.} We will only give the details for $\mu<1$. The argument works also
in the case $\mu>1$ observing that if $G(f)=f|f|^\mu$ then 
$G(0)=\dots=G^{(d)}(0)=0$ for $d=[\mu]$, being $[\mu]$ the integer part of $\mu$.\\
Because of $\sigma>1$ we can write $\sigma=1+s$ with $0<s<1$, then by the definition of the
$H_y^{\sigma}$-norm we arrive at
\begin{align}\label{eq.LF1}
\norm{f|f|^\mu}{H_y^{\sigma}}\\
\leq C\norm{f|f|^\mu}{H_y^{s}}+C \norm{\nabla_y(f|f|^\mu)}{H_y^{s}}\leq C\norm{f|f|^\mu}{H_y^{s}}+C \norm{|f|^\mu\nabla_yf}{H_y^{s}}.
\nonumber
\end{align}
Then we can see that it is enough to carry on with the last term in the previous estimate
\eqref{eq.LF1}, that is
\begin{align}\label{eq.LF2}
\norm{|f|^\mu\nabla_yf}{H_y^{s}}\\
\leq C\norm{\nabla_yf}{H_y^{s}}\norm{f}{L_y^{\infty}}^\mu +C \norm{\nabla_yf}{L_y^{2\sigma}}\norm{|f|^\mu}{W_y^{s, 2\frac \sigma s}},
\nonumber
\end{align}
where we used the first estimate of Theorem 27 in \cite{CRT}.
Since the interpolation bound 
\begin{equation}\label{eq.LF2a}
\norm{\nabla_yf}{L_y^{2\sigma}}\leq C \norm{f}{H_y^{\sigma}}^{\frac 1\sigma}\norm{f}{L_y^{\infty}}^{1-\frac 1\sigma},
\end{equation}
(see again Theorem 27 in \cite{CRT}) we need only to prove the following one
\begin{equation}\label{eq.LF3}
\norm{|\nabla_y|^s|f|^\mu}{L_y^{2\frac \sigma s}}\leq C \norm{|\nabla_y|^{\overline\sigma} f}{L_y^{2\frac \sigma  {\overline \sigma}}}^{\frac s{\overline \sigma}}\norm{f}{L_y^{\infty}}^{\mu-\frac s{\overline \sigma}} ,
\end{equation}
for some $\frac s \mu<\overline\sigma <1$.
Assume the above \eqref{eq.LF3} true, then by the interpolation estimate (we refer to the Propositions 31, 32 
 in \cite{CRT})
\begin{equation}\label{eq.LF3a}
\norm{|\nabla_y|^{\overline \sigma}f}{L_y^{2\frac \sigma{\overline \sigma} }}\leq C \norm{|\nabla_y|^{\sigma} f}{L_y^{2}}^{\frac {\overline \sigma}\sigma}\norm{f}{L_y^{\infty}}^{1-\frac {\overline \sigma}\sigma},
\end{equation}
one achieve 
\begin{align}\label{eq.LF4}
\norm{|\nabla_y|^s|f|^\mu}{L_y^{2\frac \sigma s}}\leq C \round{\norm{|\nabla_y|^{\sigma} f}{L_y^{2}}^{\frac {\overline \sigma}\sigma}\norm{f}{L_y^{\infty}}^{1-\frac {\overline \sigma}\sigma}}^{\frac s{\overline \sigma}}\norm{f}{L_y^{\infty}}^{\mu-\frac s{\overline \sigma}} \\
\leq C \norm{|\nabla_y|^{\sigma} f}{L_y^{2}}^{\frac s{ \sigma}}\norm{f}{L_y^{\infty}}^{\frac s{\overline \sigma}-\frac s\sigma}\norm{f}{L_y^{\infty}}^{\mu-\frac s{\overline \sigma}}\leq C \norm{ f}{H_y^{\sigma}}^{\frac s{ \sigma}}\norm{f}{L_y^{\infty}}^{\frac s{\overline \sigma}-\frac s\sigma}\norm{f}{L_y^{\infty}}^{\mu-\frac s{\overline \sigma}}.
\nonumber
\end{align}
Therefore  \eqref{eq.LF2a} in connection with \eqref{eq.LF4}, recalling again the definition of the
$H_y^{\sigma}$-norm, brings to
\begin{align}\label{eq.LF5}
 \norm{\nabla_yf}{L_y^{2\sigma}}\norm{|f|^\mu}{W_y^{s, 2\frac \sigma s}}\leq C\norm{f}{H_y^{\sigma}}\norm{f}{L_y^{\infty}}^{\mu}.
 \end{align}
 Finally combining \eqref{eq.LF1}, \eqref{eq.LF2} and \eqref{eq.LF5} one arrive at \eqref{eq.LF}.
 It remains to consider the estimate \eqref{eq.LF3}. Since the function $|f|^\mu$ is H\"older continuous of order $0<\mu<1$, one is in position to apply Lemma \ref{HC} with $q=q_2=2\frac \sigma s$, $q_1=\infty$ and $\sigma=\overline \sigma$ getting 
 \begin{equation*}
\norm{|\nabla_y|^s|f|^\mu}{L_y^{2\frac \sigma s}}\leq C \norm{|\nabla_y|^{\overline\sigma} f}{L_y^{2 \frac \sigma s \frac s{\overline \sigma}}}^{\frac s{\overline \sigma}}\norm{|f|^{\mu-\frac s{\overline \sigma}} }{L_y^{\infty}},
\end{equation*}
that is the desired \eqref{eq.LF3}.
\end{proof}
\begin{cor}
The same conclusions of Proposition \ref{LF} remain valid if one replace the function $f|f|^\mu$
by $|f|^{1+\mu}.$
\end{cor}

\end{document}